\documentclass[11pt]{amsart}

\usepackage{enumitem}
\usepackage{lineno} 
\usepackage{multicol} 
\usepackage[colorlinks,allcolors=blue]{hyperref}
\usepackage[nameinlink,noabbrev]{cleveref}
\usepackage{amsthm,amssymb,mathrsfs,pstricks,booktabs}
\usepackage{mathtools,amsmath,fancyhdr,appendix,nccmath}
\usepackage[all]{xy} 

\usepackage{array}
\usepackage[left=2cm,top=2cm,right=2cm,bottom=2cm]{geometry}

\usepackage{hyperref}
\hypersetup{
    colorlinks=true,
    linkcolor=teal,
    urlcolor=teal,
    citecolor=teal,
    hyperindex=true,
    bookmarks=true,
    pagebackref=true
}

\newtheorem{theorem}{Theorem}[section]

\newtheorem{corollary}[theorem]{Corollary}
\newtheorem{proposition}[theorem]{Proposition}

\theoremstyle{definition}
\newtheorem{definition}[theorem]{Definition}
\newtheorem{example}[theorem]{Example}
 
\newtheorem{remark}[theorem]{Remark}
\newtheorem{problem}[theorem]{Problem}

 \usepackage{mathrsfs}  
 \makeatletter
\def\proof{\par\pushQED{\qed}\normalfont\topsep6\p@ \trivlist \item[\hskip 20pt \itshape\proofname\@addpunct{.}\hskip\labelsep] \ignorespaces}
\makeatother

\def\fork{\mathrel{\raise0.2ex\hbox{\ooalign{\hidewidth$\vert$\hidewidth\cr\raise-0.9ex\hbox{$\smile$}}}}}
\renewcommand{\emptyset}{\varnothing}
\renewcommand{\L}{\mathcal{L}}

\newcommand{\p}{\mathfrak{p}}
\newcommand{\cl}{{\bf\mathrm{cl}}}
\renewcommand{\a}{\mathfrak{a}}
\newcommand{\m}{\mathfrak{m}}

\newcommand{\Z}{\mathbb{Z}}
\newcommand{\ev}{_{\overline{0}}}
\newcommand{\od}{_{\overline{1}}}
\renewcommand{\ker}{\mathrm{Ker}}

\usepackage[mathscr]{euscript}
\renewcommand{\qed}{\hfill$\square$}
\renewcommand{\to}
{\longrightarrow}
\newcommand{\rank}{\mathrm{rank}}

\renewcommand{\mapsto}{\longmapsto}

\renewcommand{\O}{\mathcal{O}}

\usepackage{dsfont}

\newcommand{\J}{\mathfrak{J}}

\newcommand{\q}{\mathfrak{q}}

\renewcommand{\b}{\mathfrak{b}}

\usepackage{setspace}

\newcommand{\Y}{\mathrm{supp}}
\newcommand{\U}{U}
\newcommand{\K}{\Bbbk}
\newcommand{\Zar}{ZR}

\title{Valuations on superrings} 

\author{Pedro  Rizzo}
\address{Instituto de Matemáticas\\ FCEyN\\Universidad de Antioquia}

\email{pedro.hernandez@udea.edu.co}

\author{Joel Torres del Valle} 
\email{joel.torres@udea.edu.co}

\author{Alexander Torres-Gomez} 
\email{galexander.torres@udea.edu.co}
\date{\today}

\begin{document} 

\maketitle

\begin{abstract}   
A valuation theory for superrings is developed, extending classical constructions from commutative algebra to the $\mathbb Z_2$-graded and supercommutative setting. We define valuations on superrings, investigate their fundamental properties, and explore the construction of Zariski-Riemann superspaces.
\end{abstract} 
 

\section*{Introduction}

The theory of valuations for fields is well-established and exhibits significant applications in number theory and algebraic geometry. Classical works such as \cite{Bou} and \cite{SZ} provide foundational insights into this theory. Furthermore, valuation theory has been extended beyond fields to encompass commutative rings, which allow for the existence of zero divisors, see \cite{BL, Bra, KL, Man1}. Within this broader framework, many classical results from the field setting, such as approximation theorems and the fundamental inequality, admit analogous formulations. 

The scope of valuation theory has further expanded to include non-commutative rings, as evidenced by works such as \cite{Pir} and \cite{Sch}. The continued development of valuation theory underscores its dynamic and significant nature within mathematical research. For example, \cite{Gu} develops a valuation theory for modules. Furthermore, the utility of valuations is highlighted by their role in proving cases of a conjecture by Berger \cite{Ber}, which asserts that for a one-dimensional complete local reduced $k$-algebra $(R, \m_R, k)$ over a field of characteristic zero, $R$ is regular if and only if the universally finite module of differentials $\Omega_R$ is torsion-free, as shown in \cite{MM}. Additionally, \cite{BJ} utilizes valuations to study the log canonical threshold and the stability threshold of a big line bundle $L$ on a normal complex projective variety $X$ with at most klt singularities. 

Zariski-Riemann spaces establish an important connection between valuations and algebraic geometry. Specifically, given a field extension $K \mid k$, one considers the set $\Zar(K, k)$ of valuation rings associated with non-trivial valuations on  $K$ that are non-negative on $k$. Pioneered by Zariski for their application in desingularization processes \cite{Z}, these objects have proven to be highly versatile in various applications and are, in fact, a fundamental component of Nagata's compactification theorem \cite{Nag}. Moreover, relative versions of these spaces are employed in the birational study of rigid geometry and generalized forms of de Jong’s theorem for the stable reduction of curves. For a broader overview of their applications in algebraic geometry, algebraic number theory, and other fields, we recommend \cite{PM} and the references therein.

The continued importance of Zariski-Riemann spaces in contemporary research is evident in studies such as \cite{O}. In this work, the Zariski–Riemann space $\Zar(F, D)$, where $F$ is the quotient field of a domain $D$, is analyzed. The authors explore the Zariski topology, the inverse topology (dual on a spectral space), and the patch topology (a Hausdorff refinement) on $\Zar(F, D)$ by viewing it as a projective limit of integral schemes with function fields contained in $F$, and they characterize the locally ringed subspaces of $\Zar(F, D)$ that are affine schemes. With respect to the patch topology, \cite{Spi} investigates isolated points of $\Zar(F, D)$, providing a complete characterization of when $F$ is isolated and when a valuation domain of dimension one is isolated. Consequently, they identify all isolated points of $\Zar(F, D)$ when $D$ is a Noetherian domain. Moreover, \cite{PS} focuses on pseudo-convergent sequences $E$ in $F$, studying two constructions that associate a valuation domain of $K(\Zar(F, D))$ lying over $D$. Additionally, \cite{HKK} examines cdh-descent using the functorial Riemann–Zariski space. Finally, the concept of relative Riemann–Zariski spaces is explored in \cite{T} within the context of a morphism of schemes, offering a generalization of the classical Riemann–Zariski space of a field. For a comprehensive and recent overview, \cite{O1} serves as an excellent introduction to the subject of Zariski-Riemann spaces.

While significant progress has been made in valuation theory, an analogous framework for superrings has yet to be established. The potential for extending the well-established constructions of commutative algebra into the $\mathbb{Z}_2$-graded and supercommutative realm remains an open area of investigation. This article seeks to address this gap by developing a valuation theory for superrings, with the aim of broadening the scope of valuation theory and contributing to the ongoing generalization of commutative structures to their ``super" counterparts. 

This article focuses on defining and investigating valuations on superrings. The primary aim is to advance this theory to a level where it can be utilized to address issues relevant to algebraic supergeometry, even in cases where specific applications are not elaborated upon within this work. While many of the proofs bear resemblance to their commutative counterparts (e.g., as found in \cite{IB, Man, Man1}, and \cite{SZ}), we have nonetheless chosen to include certain underlying concepts for the sake of completeness. Moreover, in numerous instances, the specific role of supercommutativity in obtaining these results will be clarified, highlighting its distinction from the use of ordinary commutativity.

We investigate whether the construction of Zariski-Riemann spaces can be generalized to the ``super" setting. To this end, for an extension of superfields $K\mid k$, we consider the topological space $X$ comprising the collection of valuation rings associated with non-trivial valuations on $K$ that are non-negative on $k$. We naturally equip $X$ with a structural sheaf, ensuring that the superreduced space recovers the classical information. Nevertheless, in the elementary case where $K\od\neq 0$ and $\overline{K}=k(x)$ is a field of rational functions over $k$, this construction does not yield the notion of an \textit{abstract non-singular supercurve}. The underlying reason for this failure is that in the ``super" context, the concepts of \textit{normality} and \textit{regularity} diverge when the even Krull superdimension of the superdomain is one (for details, see \cite[Remark 5.13 i)]{RTT}), unlike the commutative setting where they coincide for one-dimensional domains. Despite this, the theory developed herein and the resulting Zariski-Riemann superspaces seem to offer a promising approach to the exploration of birational supergeometry and its modeling with Zariski-Riemann superspaces, following the recent advancements in \cite{PM}.

The remainder of this paper is organized as follows: 

\Cref{SEC:1} lays the groundwork for our subsequent analysis by introducing the essential definitions and results from the theory of superrings.

\Cref{SEC:2} delves into the concept of valuation over a superring. Within this section, we derive several of its fundamental properties and give particular attention to the equivalence among the notions of valuation, valuation pair, and maximal pair. Furthermore, we explore the notion of $v$-convex ideals and extend the established correspondence between the isolated subgroups of a valuation's value group and the $v$-convex superideals of its valuation rings to the super setting (\Cref{Proposition:3.14}).

\Cref{SEC:4} focuses on the concept of dominance between valuations. This section also examines extensions of valuations and the property of being integrally closed for valuation superrings (\Cref{Proposition:4.12}). Additionally, we prove that if $R$ is integrally closed in $S$, then any valuation on $R$ can be extended to $S$ (\Cref{prop.3.13.ct}), and we investigate conditions under which extensions of valuations are unique up to a certain notion of equivalence (\Cref{Proposition:4.11}). 

\Cref{SEC:6} is dedicated to discussing further properties of valuations, including approximation theorems and the fundamental inequality. 

\Cref{SEC:7} introduces the definition of the Zariski-Riemann space in the super setting and provides a comparison with the purely even case. Finally, this section also outlines potential avenues for future research and some geometric implications of the present work.


\section{Basic definitions}\label{SEC:1}
 

In this section, we introduce fundamental definitions and properties of superrings, which are essential for the development of this work. Our exposition on this subject is concise, and for a more thorough examination of the elementary theory of superrings, refer to \cite[Chapters 1-4]{W}.

\subsection{Superrings}  

In this paper $\Z_2$ represents the group $\Z/2\Z$. Any ring $R$ in this paper is assumed to be unitary and non-trivial, i.e., $1\neq 0$.

\begin{definition}
    A $\Z_2$-graded ring $R=R\ev\oplus R\od$ is called a \emph{superring}. 
\end{definition}

Let $R$ be a superring. For any $x\in R$, we define the \textit{parity} of $x$ as

\[
|x|:=\begin{cases}
    0& \text{ if }x\in R\ev,\\
    1& \text{ if }x\in R\od.
\end{cases}
\]

The set $h(R):=R\ev\cup R\od$ is called the \textit{homogeneous elements of $R$} and each $x\in h(R)$ is said to be  \textit{homogeneous}.  A homogeneous nonzero element $x$ is called \emph{even} if $x\in R\ev$ and \textit{odd} if $x\in R\od$.  

Any superring $R$ in this work will be \textit{supercommutative}, i.e.,  $R\ev$ is central in $R$ and any element in $R\od$ squares to zero. This implies that  

\begin{equation}\label{Supercommutativity}
    xy=(-1)^{|x||y|}yx,\quad \text{ for all }\quad x, y\in h(R);
\end{equation}

\noindent and \eqref{Supercommutativity} is equivalent to supercommutativity if 2 is a non-zerodivisor in $R$. Observe that if $R$ is supercommutative, then $R\ev$ is a commutative ring. 

\begin{definition}    
Let $R$ and $S$ be superrings. A \textit{morphism} $ \phi :R\to S$ is a ring homomorphism \emph{preserving the parity}, that is, such that $ \phi (R_{i})\subseteq S_{i}$ for all $i\in\Z_2$. A morphism $ \phi :R\to S$ is an \textit{isomorphism} if it is bijective. In this case, we say that $R$ and $S$ are \emph{isomorphic superrings}, which we denote by $R\cong S$. 
\end{definition}

\begin{definition}
Let $R$ be a superring.

\begin{enumerate} 

\item[i)] A \textit{superideal} of $R$ is an ideal $\a$ of $R$ such that  $\a=(\a\cap R\ev)\oplus(\a\cap R\od).$

\item[ii)] An ideal $\p$ of $R$ is \textit{prime} (resp. \textit{maximal}) if $R/\p$ is an integral domain (resp. a field).
\end{enumerate}
\end{definition}

\begin{remark}\label{rmk:prime-ideal} Let $R$ be a superring. 

\begin{enumerate} 
\item[i)] If $\a$ is a superideal of $R$, we denote $\a_{i}:=\a\cap R_{i}$ for each $i\in\Z_2$.
\item[ii)]  Any superideal $\a$ is a two-sided ideal.
\item[iii)] If $\a$ is a superideal of $R$,  then the quotient ring $R/\a$ is a superring with $\Z_2$-gradding given by $R/\a=(R\ev/\a\ev)\oplus(R\od/\a\od).$ 
\item[iv)] Any prime ideal is completely prime \cite[Lemma 4.1.2]{W}.  
\item[v)] Any prime ideal of $R$ is a superideal of the form $\p=\p\ev\oplus R\od$, where $\p\ev$ is a prime ideal of the ring $R\ev$ \cite[Lemma 4.1.9]{W}.
\item[vi)] Any maximal ideal of $R$ is of the form $\m=\m\ev\oplus R\od$, where $\m\ev$ is a maximal ideal of $R\ev$.
\end{enumerate}
\end{remark}

\begin{definition}\label{def.sev.cond} Let $R$ be a superring. 

\begin{itemize}
    \item[i)] The ideal $\J_R=R\cdot R\od=R\od^2\oplus R\od$ is called the \textit{canonical superideal of $R$}.
    \item[ii)] The \textit{superreduced} of $R$ is the commutative ring $\overline{R}=R/\J_R\cong R\ev/R\od^2$. 
    \item[iii)] $R$ is called a \textit{superdomain} if $\overline{R}$ is a domain or, equivalently, $\J_R$ is a prime ideal. 
    \item[iv)] $R$ is called a \textit{superfield} if $\overline{R}$ is a field or, equivalently, $\J_R$ is a maximal ideal.
    \item[v)] We say that $R$ is \textit{split} if $R=\overline{R}\oplus\J_R$.
\end{itemize}
\end{definition}

\begin{definition} A superring $R$ is called \emph{Noetherian} if $R$ satisfies  the ascending chain condition on superideals, or, equivalently, any superideal $\a$ of $R$ is \emph{finitely generated}, i.e., there exist finitely many homogeneous elements $a_1,\ldots, a_n\in \a$  such that         $\a=Ra_1+\cdots+Ra_n.$
\end{definition}

\begin{definition}
    Let $R$ be a superring and $\U\subseteq R\ev$ a multiplicative set. \emph{The localization $\U^{-1}R$ of $R$ at $\U$}, is the superring $\U^{-1}R:=(\U^{-1}R\ev)\oplus(\U^{-1}R\od)$, where $R\od$ is regarded as $R\ev$\,-module. An element in $\U^{-1}R$ is denoted by $x/y$ or $y^{-1}x$, where $x\in R$ and $y\in \U$. If $\U=R\ev \setminus \p\ev$, where $\p$ is a prime ideal, $\U^{-1}R$ is often written as $R_\p$. If $\U$ is the set of non-zerodivisors of $R\ev$, the superring $\U^{-1}R$ is denoted by  $K(R)$ and is called \emph{the total superring of fractions of $R$}.  
\end{definition}

\begin{definition}
A superring $R$ is said to be  \emph{local} if it has a unique maximal ideal. 
\end{definition} 

\subsection{Krull's Lemma}

\begin{proposition}[``Super'' Krull's Lemma]\label{Super:Krull}
    If $R$ is a superring, $\q$ is a prime ideal of $R$ and $\a$ is any superideal of $R$ not intersecting $\U=R\ev\setminus\q\ev$, then there exists a prime ideal $\p$ of $R$ containing $\a$ and not intersecting $\U$ and neither $R\setminus\q$.
\end{proposition}

\begin{proof}
    Since $\a\ev\cap \U=\emptyset$, it suffices to find a prime ideal $\p\ev$ of $R\ev$ such that $\a\ev\subseteq\p\ev$ and $\p\ev\cap \U=\emptyset$, because in a such case, the prime ideal  $\p=\p\ev\oplus R\od$ is such that $\a\subseteq\p$ and, moreover $\p\cap \U=\emptyset$. The existence of $\p\ev$ is given by Krull's Lemma \cite[p.253]{J}. 
\end{proof}

\subsection{Supermodules} 

From now on, a ring $R$ is always meant to be a superring.

\begin{definition}
     A left $\Z_2-$graded $R-$module $M=M\ev\oplus M\od$ is called an $R-$\emph{supermodule}.
\end{definition} 

Let $M$ be an $R-$supermodule. The \textit{set of homogeneous elements of $M$} is $h(M):=M\ev\cup M\od$, and the \textit{parity} of $m\in h(M)$ is defined in the obvious way. If $l$ is a left action of $R$ on $M$, we define a right action $r$ of $R$ on $M$ by

\[
r(m, a)=(-1)^{|a||m|}l(a, m),\quad \text{ for all \quad $x\in h(R)$ and $m\in h(M)$}.
\]

\noindent These  actions commute. Thus, $M$ is both a right and a left $R-$supermodule. Hereafter any  supermodule is assumed to have these two compatible structures. 

\begin{definition}
    An $R$-supermodule $M$ is said to be \textit{finitely generated} if one can find finitely many homogeneous elements $f_1, \ldots, f_n\in M$ such that $M$ has a decomposition of the form $M=Rf_1 +\cdots +Rf_n.$ 
\end{definition}

\subsection{Integrally closed superrings} 

Let $R, S$ be superrings. In this subsection, if we write $R\subseteq S$ we mean that they have the same unit and that the $\Z_2$-gradings are compatible: $R_i=R\cap S_i$, for all $i=\overline{0}, \overline{1}$.  

\begin{definition}\label{Def:1.12}
    Let $R\subseteq S$ be superrings.
    
    \begin{itemize}[itemsep=2pt]
        \item[i)] A homogeneous $x\in S$ is said to be \textit{integral over} $R$ if for some unitary finitely generated $R$-supermodule, say $M=Rx_1+\cdots+ Rx_n$, where $x_1, \ldots, x_n$ are even elements in the center $\mathcal{Z}(S)$ of $S$, we have $x M\subseteq M$.     

        \item[ii)] An arbitrary  $x\in S$ is  \textit{integral over} $R$ if its homogeneous components are so.

        \item[iii)] The \textit{integral closure of $R$ in $S$}, $\cl_S(R)$, is the set of elements in $S$ that are integral on $R$. If $S=K(R)$, $\cl_S(R)$ is written by $\cl(R)$.

        \item[iv)] The superring $R$ is said to be \textit{integrally closed in} $S$ if $R=\cl_S(R)$.

        \item[v)] If $R$ is  integrally closed in $S=K(R)$, we say that $R$ is \textit{normal}. 
    \end{itemize}
\end{definition} 

It is not hard to see that $\cl_S(R)$ is a superring containing $R$ and that $\cl_S(R)$ is integrally closed in $S$. 

\begin{proposition}\label{Proposition:2.12} If $x\in S$ is integral over $R$ and commutes with all of $R$, then $x$ satisfies an equation of the form

    \[
    x^n+\sum_{0\leq i\leq n-1}a_ix^{i}=0, \quad  \text{where} \quad a_0, a_1, \ldots, a_{n-1}\in R.
    \]

    \noindent Conversely, if $x$ is an element in the center of $S$ satisfying an equation of this form, then $x\in \cl_S(R)$.  
\end{proposition}

\begin{proof}
   See \cite[Proposition 5.4]{RTT}.
\end{proof}

\subsection{Ordered groups} 

In this paper any group is assumed to be abelian and with additive structure. 

\begin{definition}
    Let $\mathbb{G}$ be a group. We say that $\mathbb{G}$ is \textit{totally  ordered} or \textit{ordered} if there is a total order relation $\leq$ in $\mathbb{G}$, which is  invariant under translation. In other words, for any $x, y, z\in\mathbb{G}$, the relation $x\leq y$ implies that $z+x\leq z+y$.
\end{definition}

\begin{definition}
    Let $\mathbb{G}$ be a totally ordered group, and $\mathbb{H}\subseteq \mathbb{G}$.

    \begin{itemize} 
        \item[i)] We say that $\mathbb{H}$ is a \textit{segment} if it has the following property: if  $x\in\mathbb{G}$ belongs to $\mathbb{H}$, then for all $y\in\mathbb{G}$, $-x\leq y\leq x$ implies  $y\in\mathbb{H}$.
        \item[ii)] If $\mathbb{H}$ is a proper subgroup of $\mathbb{G}$ which is also a segment, it is called an \textit{isolated subgroup} of $\mathbb{G}$.
    \end{itemize}
\end{definition}

We write $\mathcal{S}(\mathbb{G})$ and  $\mathcal{I}(\mathbb{G})$ for the set of segments and isolated subgroups of $\mathbb{G}$, respectively. It is a folklore that if $\mathbb{G}$ is an ordered group, then $(\,\mathcal{S}(\mathbb{G}), \subseteq\,)$ and $(\,\mathcal{I}(\mathbb{G}), \subseteq\,)$ are total orders. 

\begin{remark}
    Let $\mathbb{G}$ be a group. We can add $\mathbb{G}$ a symbol $ \infty$ under the conventions $$\infty+\infty=\infty+\alpha=\alpha+\infty=\infty$$ and $\infty>\alpha$ for all $\alpha\in\mathbb{G}$. We keep $\infty\,-\,\infty$ undefined.
\end{remark}


\section{Valuations and Valuation Pairs}\label{SEC:2} 


In this section we define valuations and valuation pairs and prove their basic properties.  

\subsection{Valuations}

\begin{definition}\label{Def:of:valuation}
 
     A \textit{valuation}\index{valuation} on a superring $R$ is a pair $(v, \mathbb{G}\cup\{\infty\})$, where $(\mathbb{G}, +, 0)$ is an ordered abelian group,  and $v:R\to\mathbb{G}\cup\{\infty\}$ is a \textit{surjective} map such that the following conditions are satisfied 
 
     \begin{itemize}[itemsep=2pt]
        \item[i)] $v(0)=\infty$   and $v(1)=0$. 
        
        \item[ii)] $v(xy)=v(x)+v(y)$ for all $x, y\in R$. 
        
        \item[iii)]   $v(x+y)\geq\min\{v(x), v(y)\}$ for all $x, y\in R$.
    \end{itemize}

    \noindent We say that $v$ is \textit{trivial} if $\mathbb{G}=0$. Otherwise, $v$ is said to be \textit{nontrivial}. \index{valuation!trivial}\index{valuation!nontrivial}  
\end{definition}

Henceforth, a ring $R$ is always meant to be a superring and $v$ is a valuation on $R$ which is assumed to take values in a group $\mathbb{G}$. Thus, the pair $(v, \mathbb{G}\cup\{ \infty\})$ is often denoted simply by $(v, \mathbb{G})$ or $v$. A valuation $w$ (on any given superring) is always assumed to take values in a group $\mathbb{H}$. For instance, we may write $\mathbb{G}_u$ for the value group of a valuation $u$, and $\mathbb{H}_i$ for the value group of $w_i$, etc.

\begin{remark}
    \ 

\begin{itemize}
    \item[i)] Observe that we required for valuations the condition of being surjective. This attribute gives them the surname ``Manis'' in the commutative setting (cf. \cite{KnK, KZ, Man1}). We will use the term ``valuation'' throughout, with the understanding that we are considering only surjective valuations.  
    \item[ii)] Let $v$ be a valuation on $R$. Note that if $x\in R$, then $v(-x)=v(x)$ and if $x$ is also invertible, then $v(x^{-1})=-v(x)$ and if $v(x)\neq v(y)$, then $v(x+y)=\min\{v(x), v(y)\}$. More generally, if $x_1, \ldots, x_n\in R$, then 
    \[
    v(x_1+\cdots+x_n)\geq\min\{v(x_1), \ldots, v(x_n)\}
    \]
    with equality holding only when the minimum is uniquely attained by an $x_i$ (cf. \cite[p.33]{SZ}).
    \end{itemize}
\end{remark}

\begin{proposition}\label{P:v.sends.J.to.infinity}
    Let $v$ be a valuation on $R$. Then, $v(\J_R)=\{\infty\}$.
\end{proposition}

\begin{proof}
    If $x\in R\od$, then $x^2=0$ and $\infty=v(0)=v(x^2)=v(x)+v(x)$, so $v(x)=\infty$. Let $x=x\ev+x\od\in\J_R$. Since $v(x\od)=\infty$, we have $v(x)=v(x\ev)$. Observe that, $x\ev=s_1r_1+\cdots+s_nr_n$, for some $s_i, r_i\in R\od$. Then, $v(x)=v(x\ev)=v(s_1r_1+\cdots+s_nr_n)\geq\min\{v(s_1r_1), \ldots, v(s_nr_n)\}=\infty.$ Thus, $v(x)=\infty$ for all $x\in\J_R$.
\end{proof}

\subsection{Valuation pairs} Let $v$ be a valuation on $R$. It is not hard to see that $$A_v:=\{x\in R\mid v(x)\geq0\}$$ is a subsuperring of $R$ and $$\p_v:=\{x\in R\mid v(x)>0\}$$ is a prime superideal of $A_v$ containing $\J_R$. This motivates the following definition of valuation pair. 
 
\begin{definition}\label{Def:ValuationPair}
By a \textit{valuation pair of}  $R$, we mean a pair $(A, \p)$ such that:\index{valuation pair}
\begin{itemize}
        \item[i)] $A$ is a subsuperring of $R$ and $\p$ is a prime ideal of $A$ (containing $\J_R$).
        \item[ii)] For all $x\in R \setminus A$, there exists $x'\in\p\ev$ such that $xx'\in A \setminus \p$.
    \end{itemize}
\end{definition} 

\begin{proposition}\label{Theorem:3.4}
\it Let $(A, \p)$ be a valuation pair of $R$. Then, there is a valuation $v$ on $R$ such that

\begin{equation}\label{eq.th:3.4}
    A=A_v\quad\text{and}\quad\p=\p_v.
\end{equation}
Conversely, the pair $(A_v, \p_v)$ associated to a valuation $v$ on $R$ is a valuation pair of $R$.
\end{proposition}

\begin{proof} Let $x, y\in R$ and $(x:\p)_{R\ev}:=\{z\in R\ev\mid zx\in\p\}.$ We define an equivalence relation 

\begin{center}
    $x\sim y$ if $(x:\p)_{R\ev}=(y:\p)_{R\ev}$.
\end{center}
The symbol $v(x)$ stands for the equivalence class of $x\in R$, and we define $0:=v(1)$ and $ \infty:=v(0)=R\ev$. Let $\mathbb{G}\cup\{ \infty\}:=R/\sim$ be the set of equivalence classes, where $\mathbb{G}:=(R/\sim) \setminus \{ \infty\}$. We define the addition $v(x)+v(y):=v(xy)$, and the order 
\begin{center}
    $v(x)\geq v(y)$ if $v(x)=v(y)$ or $v(x)>v(y)$,
\end{center}
where 
\begin{center}
    $v(x)>v(y)$ if there is some $z\in R\ev$  with $zx\in\p$ and $zy\in A-\p$,
\end{center}
i.e., $v(x)>v(y)$ if and only if  $(x:\p)_{R\ev}-(y:\p)_{R\ev}\neq\emptyset$. This order is  total and invariant by translation. Further, $ \infty>\alpha$ for all $\alpha\in\mathbb{G}$. It follows that $(v, \mathbb{G})$ is a valuation and \eqref{eq.th:3.4} holds. The converse is immediate.
\end{proof}

Hereafter, if $\mathbb{G}$ and $\mathbb{H}$ are ordered groups and $h:\mathbb{G}\to\mathbb{H}$ is a group homomorphism preserving the order, we simply say that it is an \textit{order homomorphism}.\index{order homomorphism} We also use the notation  $\mathbb{G}_\infty:=\mathbb{G}\cup\{ \infty\}$ and extend $h$ to get   $h:\mathbb{G}_\infty\to\mathbb{H}_\infty$ as $h(\alpha)= \infty$ if and only if $\alpha= \infty$. However, we maintain the notation $h:\mathbb{G}\to\mathbb{H}$ for it, except when clearness is at risk. 

\begin{remark}\label{REMAKR:6.1.6}
    Let $v$, $w$ be valuations on $R$. We say that they are \textit{equivalent}\index{valuations!equivalent} if there is an order isomorphism $h:\mathbb{G}\to\mathbb{H}$ such that $h\circ v=w$. It is easy to see that this defines an equivalence relation on the set of valuations on $R$. Now, if $v$ and $w$ are equivalent valuations on $R$, then $v$ and $w$ induce the same valuation pair; that is, $(A_w, \p_w)=(A_v, \p_v)$. Conversely, the valuation induced on $R$ by $(A_v, \p_v)$, as in the proof of \Cref{Theorem:3.4}, is equivalent to $v$ (cf. \cite[pp.33-34]{SZ}). We are then allowed to speak of \textit{the  valuation $v$ determined by a valuation pair} $(A, \p)$ and of \textit{the valuation pair $(A_v, \p_v)$ determined by a valuation} $v$. 
\end{remark}

\begin{proposition}\label{Proposition:3.6} \it Let $v$ be a valuation on $R$. If $x\in R \setminus A_v$ is invertible, then $x^{-1}\in A_v$.
\end{proposition}

\begin{proof}
    Since $x\in R \setminus A_v$, then $v(x)<0$, so $v(x^{-1})>0$ and therefore $x^{-1}\in A_v$. 
\end{proof}

\subsection{Maximal pairs} Let $\mathcal{C}$ be the set of pairs $(A, \p)$ such that $A$ is a subsupering of $R$ and $\p$ is a prime ideal of $A$ containing $\J_R$. We define an order relation in $\mathcal{C}$ in the following way
    
    \vspace{0.2cm}

    \begin{equation}\label{Eqn:Zorn}
    (A, \p)\preceq(B, \q)\quad \text{ if } \quad A\subseteq B \text{ and }\p=A\cap\q.
    \end{equation}

    \vspace{0.2cm}

    Thus, \eqref{Eqn:Zorn} defines a partial inductive order in $\mathcal{C}$. By Zorn's Lemma,  $\mathcal{C}$ has at least one $\preceq$-maximal element, which in the sequel is called a \textit{maximal pair}.
\medskip

In the realm of commutative rings (see e.g., \cite[Proposition 
 1 (i), (ii)]{Man1}),  valuation are maximal pairs  and \textit{viceversa}. This is also true in the setting of the present work, as the following proposition shows.\index{maximal pair}

\begin{proposition}
\it    A pair $(A, \p)$ is a valuation pair if and only if it is a maximal pair. 
\end{proposition}

\begin{proof} The ``only if'' part of the proposition is shown as in the purely case \cite[Proposition 1.11]{IB}. 

For the ``if'' part, consider  a maximal pair $(A, \p)$ and $x\in R\ev \setminus A\ev$ and define  $B=A[x]$ and $\q=B\p$. It follows that $\q$ is a superideal in $B$ such that $\p\subseteq(\q\cap A)$. If $\p=A\cap\q$, then also $\p\ev=A\ev\cap\q\ev$, and if $x, y\in A\ev \setminus \p\ev$ are such that $xy\in\p\ev$, then we find that $x$ or $y$ belongs to $\p\ev=A\ev\cap\q\ev$, a contradiction. Thus, $\U=A\ev \setminus \p\ev$ is a multiplicatively closed subset of $B\ev$ with $\q\cap\U=\emptyset$. By \Cref{Super:Krull}, we can find some prime ideal $\p$ of $B$ containing $\q$ and not intersecting $A \setminus \p$. In other words, $\p=\p\cap A$ and $A\subseteq B$. That is, $(A, \p)\preceq (B, \p)$, contrary to the maximality assumption on $(A, \p)$. It then follows that $\p$ is properly contained in $\q\cap A$. This means that we can find $a_i\in\p$ with  $i=0, \ldots, n$, say, and $a'\in A \setminus \p$ such that  

    \begin{equation}\label{Equation:1}
       \sum_{0\leq i\leq n}a_ix^i=a'\in A \setminus \p.
    \end{equation}

Note that \[
a'=\sum_{0\leq i\leq n}a_ix^i=\sum_{0\leq i\leq n}(a_i)\ev x^i+\sum_{0\leq i\leq n}(a_i)\od x^i.\]
Thus, since  \[\sum_{0\leq i\leq n}(a_i)\od x^i\in\J_R\subseteq\p,\] we see that  \eqref{Equation:1} is equivalent to  
     
    \begin{equation}\label{equation:2}
        \sum_{0\leq i\leq n}(a_i)\ev x^i=a''\in A \setminus \p.
    \end{equation}
     
     If $n=1$, we are done by \eqref{equation:2}, since in such a case, $(a_1)\ev x=a''-(a_0)\ev\in A \setminus \p$. If $n>1$, we proceed as in the purely even case \cite[Proposition 1.11]{IB} to find a contradiction.
\end{proof}
 
Since maximal pairs are just valuation pairs, we will always refer them by the latter name.
\medskip

To introduce some illustrative examples, we first define polynomial superrings. 

\begin{definition}
    Let $\K$ be a field. We then consider the following ring

\begin{equation}\label{eqn:kalg}
    R   =   \K[X_1, \ldots, X_s \mid  \theta_1, \ldots, \theta_d]  
       :=   \K\langle Z_1, \ldots, Z_s, Y_1, \ldots, Y_d\rangle/(Y_iY_j+Y_jY_i, Z_iZ_j-Z_jZ_i, Z_iY_j-Y_jZ_i).
\end{equation}

This $\K$-superalgebra is called the \textit{polynomial superalgebra over} $\K$, with \textit{even indeterminants} $X_i$'s and \textit{odd indeterminants} $\theta_i$'s. The element $\theta_i$ corresponds to the image of $Y_i$ and $X_j$ corresponds to the image of $Z_j$ in the quotient (\ref{eqn:kalg}). The superreduced of $R=\K[X_1, \ldots, X_s \mid  \theta_1, \ldots, \theta_d]$ is $\overline{R}\cong\K[X_1, \ldots, X_s]$ and it is not hard to check that any element of $R$ can be written in the form 

\vspace{0.2cm}

\[
f=f_{i_0}(X_1, \ldots, X_s)+\sum_{J\,:\,\text{even}}f_{i_1\cdots i_J}(X_1, \ldots, X_s)\theta_{i_1}\cdots\theta_{i_J}+\sum_{J\,:\,\text{odd}}f_{i_1\cdots i_J}(X_1, \ldots, X_s)\theta_{i_1}\cdots\theta_{i_J},
\]

\vspace{0.2cm}

\noindent where $f_{i_0}, f_{i_1\cdots i_J}\in\K[X_1, \ldots, X_s]$ for all $J$. 
\end{definition}

\begin{example}\label{Example:3.10} Let $S=Q(R):=R_{\J_R}$ be the superfield of fractions of $R=\mathbb{C}[X]_\m[\,\theta_1, \ldots, \theta_N]$, where $\m$ is a maximal ideal of $\mathbb{C}[X]$. Note that $\overline{R}=\mathbb{C}[X]_\m$ is a discrete valuation ring, so take on its field of fractions a discrete valuation $v_0$ whose valuation ring is $\overline{R}$. Consider the valuation $v:S\to\mathbb{Z}_{\infty},    x\mapsto v_0(\overline{x})$. Note that $R\neq A_v$ even when clearly $R\subset A_v$. Indeed, $A_v$ includes more elements than $R$. For example, the  odd part of $S$ is inside $A_v$, but definitively not in $R$. Now, $R$ is not integrally closed, unless it were trivial (cf. \cite[Remark 5.12]{RTT}), while the ring associated to the valuation $v$ must be, as we shall see in \Cref{Proposition:4.12}. Although $R$ is a discrete valuation superring, there exists no discrete valuation in $S$ whose valuation superring is $R$.
\end{example}

As illustrated by \Cref{Example:3.10}, the valuation ring of a discrete valuation fails to be a discrete valuation superring in the sense  \cite{RTT}. To clearly distinguish these valuation rings, we use the term \textit{realizable discrete valuation superrings} specifically for them.

\begin{example}

Let $z$ be a transcendental element over $\mathbb{C}$, $\mathbb{C}[z]$ the ring of polynomials in the variable $z$, and $\mathbb{C}(z)$ its field of fractions. For any rational function $h = \frac{f}{g} \in \mathbb{C}(z)$ and an element $z_0 \in \overline{\mathbb{C}} := \mathbb{C} \cup \{\infty\}$, we define $o(h,z_0)$ as the integer associated to the order of the zero or pole of $h$ at $z_0$, that is, the smallest integer $m$ such that
\[
h = (z - z_0)^m \frac{f_0}{g_0},
\]
where $f_0, g_0 \in \mathbb{C}[z]$ and $f_0(z_0) \neq 0$, $g_0(z_0) \neq 0$. Thus, we have a valuation 
\[
v_{z_0} : \mathbb{C}(z) \longrightarrow \mathbb{Z}_\infty
\]
\[
h \mapsto o(h,z_0)
\]
Now, if $R=\mathbb{C}(z)[\,\theta_1, \ldots, \theta_N]$, we obtain a valuation on $R$, given by $v:R\mapsto\mathbb{Z}_{\infty},
    f\mapsto v_{z_0}(\overline{f})$. Since $\mathbb{C}(z)$ is a field, then $v^{-1}( \infty)=\J_R$.  
\end{example}

\subsection{The support of a valuation}

\begin{definition}
    The \textit{support} of $v$ is the set $\Y(v):=v^{-1}( \infty)$, and the \textit{center of $v$} is $\p_v$.\index{valuation!support}
\end{definition}  

The following proposition can be readily deduced from the definition. 

\begin{proposition}\label{Proposition:3.7} \it Let $R$ be a superring and $v$ a valuation on $R$.

\begin{itemize} 
    \item[\rm i)]  $\Y(v)$ is a prime ideal of $R$.
    \item[\rm ii)]  $\J_R$  is a subset of $\Y(v)$.
    \item[\rm iii)] If $A_v\neq R$ and $\a$ is a superideal of $R$ with $\a\subseteq A_v$, then $\a\subseteq\Y(v)$. 
    \item[\rm iv)] If $v$ is nontrivial,  then $\Y(v)=\bigcap\limits_{\alpha\in\mathbb{G}}\{x\in R\mid v(x)\geq\alpha\}$.\qed 
\end{itemize}

\end{proposition}

As a consequence of \Cref{Proposition:3.7} i), given that $\Y(v)$ is a prime ideal of $R$, $R/\Y(v)$ is an integral domain. We denote its field of fractions by $k(\Y(v))$. Consider the map $$\hat{v}: k(\Y(v)) \to\mathbb{G}\cup\{ \infty\}, x/y\mapsto\overline{v}(xy'),$$ where $\overline{v}$ is the map $$\overline{v}:R/\Y(v)\to\mathbb{G}\cup\{ \infty\}, x+\Y(v)\mapsto v(x),$$ and $y'$ is such that $\overline{v}(yy')=0$.
\medskip

A straightforward computation demonstrates the following proposition.

\begin{proposition}\label{vhat:vline:are:valuations}
\it  The maps $\overline{v}$ and $\hat{v}$ are valuations. \qed 
\end{proposition}

Classical valuation theory for fields implies that $\mathfrak{o}_v:=A_{\hat{v}}$ is a local ring with maximal ideal $\m_v:=\p_{\hat{v}}$. We denote the residue field $\mathfrak{o}_v/\m_v$ by $k(v)$.  From the definition (and \Cref{vhat:vline:are:valuations}), we find that if $\mathbb{G}_+$ is the set of positive elements in $\mathbb{G}$, then $\hat{v}(\mathfrak{o}_v)=\mathbb{G}_+\cup\{ \infty\}$. 

\begin{proposition}\label{prop.2.15} \it Let $v$ and $w$ be valuations on $R$. The following conditions are equivalent. 

\begin{itemize}
    \item[\rm i)] $v\sim w$.
    \item[\rm ii)] $v(x)\geq v(y)$ if and only if $w(x)\geq w(y)$, for all $x, y\in R$.
    \item[\rm iii)] $\Y(v)=\Y(w)$ and $\mathfrak{o}_v=\mathfrak{o}_w$.
\end{itemize}
\end{proposition}

\begin{proof} The implications i) $\Rightarrow$ ii) and ii) $\Rightarrow$ iii) are direct. To show  iii) $\Rightarrow$ i), we must construct an order isomorphism $h:\mathbb{G}\to\mathbb{H}$ such that $w=h\circ v$. Let $\alpha\in\mathbb{G}$. Then, there exists $x\in R \setminus \Y(v)$ such that $v(x)=\alpha$. Hence, $h:\mathbb{G}\to\mathbb{H}; \alpha\mapsto w(x)$ is the required  order isomorphism. 
\end{proof}

\begin{proposition}
    $k(\Y(v))=(R/\Y(v))\cdot\mathfrak{o}_v^*$, where $\mathfrak{o}_v^*=\mathfrak{o}_v\setminus\m_v$.
\end{proposition}

\begin{proof} 
Routine verification. 
\end{proof}

Note that although $(\mathfrak{o}_v, \m_v)$ is always local, this is not necessarily true for $(A_v, \p_v)$. This is false even for the purely even case. For example, if $R=\Z[x, x^{-1}]$, then $(A, \p)=(\Z[x], x\Z[x]$) is a valuation pair, but $\p$ is properly contained in $\q=\p+2\Z$, which is a prime ideal, so $\p$ is not maximal \cite[p.161]{IB2}.   

\begin{definition}
    We say that $v$ is \textit{local}\index{valuation!local} if $\p_v$ is the only maximal ideal of $A_v$.
\end{definition}

\begin{proposition}\label{equiv:v:local}
\it The following conditions are equivalent. 
 
    \begin{itemize} 
        \item[\rm i)] $v$ is local.
        \item[\rm ii)] $(R, \Y(v))$ is local. 
        \item[\rm iii)] $v$ is local and $\Y(v)$ is maximal in $R$. 
    \end{itemize}
\end{proposition}
 
\begin{proof} The same as in the purely even case \cite[Chapter I, Proposition 1.3]{KZ}. 
\end{proof}

Let $\U\subseteq R\ev$ be a multiplicative set with $\U\cap\Y(v)=\emptyset$. Consider $\U^{-1}R$ the localization superring of $R$ by $U$. The map ${\scriptstyle\U^{-1}}v:\U^{-1}R\to\mathbb{G}\cup\{ \infty\}$, defined by ${\scriptstyle\U^{-1}}v(x/s):=v(x)-v(s)$, for all $x/s\,\in\,\U^{-1}R$ is a valuation. In the special case where $\U=A \setminus \p$, for a valuation pair $(A,\p)$ induced for some valuation $(v, \mathbb{G})$, we have $({\scriptstyle\U^{-1}}v)(x/s)=v(x)$ for all $x/s\,\in\,\U^{-1}R$ and $A_{\U^{-1}v}=\U^{-1}A_v=A_\p$ and $\p_{{\scriptstyle\U^{-1}}v}=\U^{-1}\p_v=\p_\p.$ In particular, ${\scriptstyle\U^{-1}}v$ is local. We call $\widetilde{v}:={\scriptstyle\U^{-1}}v$ the \textit{localization} of $v$ at $\U$.

\begin{proposition}
\it The superideal $\Y(v)$ is the only maximal ideal of $R$ among all other superideals of $R$ not meeting $A \setminus \p$.
\end{proposition}

\begin{proof}
    Since $\tilde{v}$ is local, the proof follows from  \Cref{equiv:v:local}.
\end{proof} 

\subsection{
$v$-convex superideals
}\label{SEC:3}   When working with valuation domains an important property of these is that their ideals can be totally ordered by inclusion. When extend to valuation rings this property is lost, but one can restrict oneself to $v$-convex ideals, where $v$ is the underlying valuation and here we do have a total order under inclusion. In this subsection we show that in the realm of superrings this property is preserved, and furthermore, such ideals turn out to be $\Z_2$-graded. Then we stablish some interesting properties of these superideals.

\begin{definition}
    Let $R$ be a superring and $v$ a valuation on $R$. A $v$-\textit{convex ideal} $\a$ of $A_v$ is an ideal such that $x\in \a \text{ and } v(y)\geq v(x) \text{ implies } y\in\a.$ 
\end{definition}

Manis (\cite{Man}) initially employed the term ``$v$-closed ideal'' to denote the aforementioned objects. However, in the current days, the term ``$v$-convex'' is more frequently employed (see e.g., \cite{KnK} and \cite{KZ}).

\begin{proposition}
      Any $v$-convex ideal is $\Z_2$-graded.  Furthermore, the set of $v$-convex ideals is totally ordered by inclusion. 
\end{proposition}

\begin{proof} 
    Let $\a$ be a $v$-convex ideal and consider $x\in\a$. Note that $v(x\od)=v(-x\od)= \infty\geq v(x)$, then $x\od, -x\od\in\a$, and  $x\ev=x-x\od\in\a$. Hence, $\a=\a\ev\oplus\a\od$, and any $v$-convex ideal $\a$ is a superideal. Finally, let $\a, \b$ be $v$-convex superideals of $A_v$. Suppose that there is some $x\in\b-\a$, i.e., $\b\not\subseteq\a$. Then, $v(x)<v(y)$ for all $y\in\a$. Thus,  $y\in\b$. In other words, if $\b-\a$ is nonempty, then any element of $\a$ belongs to $\b$, i.e., $\a\subseteq\b$. This completes the proof of the proposition.
\end{proof}

\begin{proposition}\label{Proposition:3.14} \it Let $(v, \mathbb{G})$ be a valuation on $R$. Then the following conditions hold. 

    \begin{itemize} 
        \item[\rm i)] There is a one to one order reversing correspondence between the set of $v$-convex superideals of $A_v$ not contained in $\Y(v)$ and the proper segments of $\mathbb{G}$.
        \item[\rm ii)] The $v$-convex prime superideals of $A_v$ not contained in $\Y(v)$ are in one to one order reversing correspondence with the proper isolated subgroups of $\mathbb{G}$.
        \item[\rm iii)] The $v$-convex prime  ideals of $A_v$ are those prime ideals $\a$ of $A_v$ with  $\Y(v)\subseteq\a\subseteq\p_v$.
    \end{itemize}  
\end{proposition}

 \begin{proof}The proofs of Conditions ii) and iii) are identical to those presented in the purely even case, as discussed in detail in \cite[Chapter 2, \S1]{IB}. To prove i), let $\a \subseteq A_v$ be a $v$-convex superideal and define the sets 
\begin{align*}
    \a^v & := \{\alpha \in \mathbb{G} \mid \alpha = v(x) \text{ for some } x \in \a  \setminus  \Y(v)\}, \\
    -\a^v & := \{\alpha \in \mathbb{G} \mid -\alpha \in \a^v\}, \\
    \mathbb{G}_\a & := \mathbb{G}  \setminus  (-\a^v \cup \a^v).
\end{align*}
Proceeding as in the proof of \cite[Chapter VI, \S10, Theorem 15]{SZ}, we show that $\mathbb{G}_\a$ is a proper segment of $\mathbb{G}$. 
On the other hand, let $\mathbb{H} \subset \mathbb{G}$ be a proper segment, and consider  
\begin{align*}
    L & = \mathbb{G}_ +  \setminus \mathbb{H}, \\
    \a_H & = L^{v^{-1}} = \{x \in A_v \mid v(x) \in L\}.
\end{align*} 
One can use the fact that $L + \mathbb{G}_ +  \subseteq L$ to show that $\a_H$ is a $v$-convex superideal in $A_v$ (for example, see the proof of \cite[Chapter VI, \S10, Theorem 15]{SZ}). The mappings $\a \mapsto \mathbb{G}_\a$ and $H \mapsto \a_H$ are seen to be order-reversing and mutually inverses, so i) follows.
\end{proof}


\section{Dominance and Extension of valuations}\label{SEC:4}


 \subsection{Dominance} Let $R$ be a superring and let $(v, \mathbb{G})$ and $(w, \mathbb{H})$ be valuations on $R$.

\begin{definition} 
    We say that $w$ \textit{dominates} $v$ if $w=h\circ v$, where $h:\mathbb{G}\to\mathbb{H}$ is an order homomorphism. In such case we write $w\gg v$ or $v\ll w$.\index{domination}  
    
    If there is some nontrivial valuation $u$ dominating $v$ and $w$, we say that $v$ and $w$ are \textit{dependent}. Otherwise we say that $v$ and $w$ are \textit{independent}. \index{valuation!dependent}\index{valuation!independent}
\end{definition}

\begin{proposition}\label{Proposition:3.17}
    \it     $w\gg v$ if and only if $\Y(w)=\Y(v)\subseteq\p_w\subseteq\p_v\subseteq A_v\subseteq A_w$.
\end{proposition}

\begin{proof} 
    The ``only if'' part  is direct. 
    
    For the ``if'' part, let $\alpha\in\mathbb{G}\cup\{ \infty\}$. Since, $v$ is surjective, there is some $x\in R$ such that $v(x)=\alpha$. We define $h(\alpha):=w(x)$. Thus, $h:\mathbb{G}\to\mathbb{H}$ is a well-defined order homomorphism such that $w=h\circ v$. Hence,  $w\gg v$ and the proof is complete. 
\end{proof} 

\Cref{Proposition:3.17} can be refined in the following way:

\begin{proposition} 
     $w\gg v$ if and only if $\Y(v)=\Y(w)$ and $\mathfrak{o}_v\subseteq\mathfrak{o}_w$. In particular, if $w\gg v$ and $v\gg w$, then $v\sim w$.
\end{proposition}

\begin{proof}  
    By \Cref{Proposition:3.17}, we may assume in any case that $\Y(v)=\Y(w)$. Then $w\gg v$ if and only if $\hat{w}\gg\hat{v}$, which in turn is seen to be equivalent to $\mathfrak{o}_v\subseteq\mathfrak{o}_w$, by the classical valuation theory for fields. 
\end{proof}

Let $\mathcal{V}_v(R)$ be the set consisting of the valuations on $R$ dominating $v$. We will show that it is linearly ordered by $\gg$. We begin by showing the following proposition.

\begin{proposition}\label{Prop.4.4} 
    Let $\p$ and $\q$ be prime ideals of $A_v$ with $\Y(v)\subseteq\p\subseteq\p_v$ and $ \Y(v)\subseteq\q \subseteq\p_v$. Then  $\p\subseteq\q$ or $\q\subseteq\p$. 
\end{proposition}

\begin{proof}
    By \Cref{Proposition:3.14} iii), the ideals $\p$ and $\q$ are $v$-convex. Thus, the proposition follows directly from the fact that $v$-convex ideals are linearly ordered by inclusion.
\end{proof}

Now, let us show that the set $\mathcal{V}_v(R)$  is linearly ordered by the relation $\gg$. Indeed,  let $u, w\in\mathcal{V}_v(R)$. By \Cref{Proposition:3.17},   $\Y(v)\subseteq\p_u \subseteq\p_v$ and $\Y(v)\subseteq \p_w\subseteq\p_v$, and by  \Cref{Prop.4.4}, $\p_u\subseteq\p_w$ or $\p_w\subseteq\p_u$. By symmetry, assume that $\p_u\subseteq\p_w$. It is not hard to see that $A_w\subseteq A_u$ and $\Y(w)\subseteq A_u$; that is, $\Y(v)=\Y(w)\subseteq\p_u\subseteq\p_w\subseteq A_w\subseteq A_v$. Thus, the statement follows from \Cref{Proposition:3.17}.
\medskip

Let $(w, \mathbb{H})\in\mathcal{V}_v(R)$. Then, there is some order homomorphism $h:\mathbb{G}\to\mathbb{H}$ such that $w=h\circ v$. The kernel of $h$, $\mathbb{I}_w:=\ker(h)$, is an isolated subgroup of $\mathbb{G}$. Consider the map $\Psi_v:\mathcal{V}_v(R)\to\mathcal{I}(\mathbb{G})$, $w\mapsto\mathbb{I}_w$, where the elements in $\mathcal{V}_v(R)$ are consider up to equivalence of valuations (see \Cref{REMAKR:6.1.6}).   

\begin{proposition}\label{Prop:3.21}
    The map $\Psi_v$ establishes a one-to-one, order-preserving correspondence between the proper isolated subgroups of $\mathbb{G}$ and the valuations on $R$ dominating $v$, considered up to equivalence. 
\end{proposition}

\begin{proof} 
    One can proceed as in the purely even case \cite[Proposition 2.7]{IB}. 
\end{proof}

We will find the following proposition useful in later discussions.
 
\begin{proposition}\label{induced:valuation}
\it If $w\gg v$ and $h:\mathbb{G}\to\mathbb{H}$ is an order homomorphism with $w=h\circ v$, then there exists a valuation $((w, v), \mathbb{G})$ on  $A_v/\p_v$ such that $(w, v)\circ\pi=v$, where $\pi:A_w\to A_w/\p_w$ is the canonical projection. Further, we have the following equalities:

\begin{itemize}[itemsep=2pt]
    \item[\rm i)] $\left(A_{(w, v)}, \p_{(w, v)}\right)=(A_v/\p_w, \p_v/\p_w)$, 
    \item[\rm ii)] $\Y({(w, v)})=\pi(\p_w)$, and
    \item[\rm iii)] $(w, v)(A_w/\p_w)=h^{-1}(0)\cup\{ \infty\}$.
\end{itemize} 
\end{proposition}

\begin{proof}  Consider any $x\in A_w \setminus \p_w$ and let $y\in x+\p_w$. We will show that $v(x)=v(y)$. Indeed, $y$ has the form $x+z$, with $z\in\p_w$. Note that $x\in A_w \setminus \p_w$ implies that $w(x)=0$, so $v(x)=0$ and $v(x)<v(z)$. Then $0=v(x)=\min\{v(x), v(z)\}=v(x+z)=v(y)$. It then follows that the map 

\vspace{-0.2cm}

\begin{eqnarray*}
    (w, v):A_w/\p_w&\to&\mathbb{G}\\
    x+\p_w&\mapsto&(w, v)(x):=v(x),
\end{eqnarray*}

\noindent is well-defined and surjective. With a straightforward calculation, one can show that $(w, v)$ defines a valuation on $A_w/\p_w$ that satisfies the desired conditions. 
\end{proof}

\subsection{Extensions of valuations}\label{SEC:5} Let $L$ be a field extension of $K$ and $v$ a valuation on $K$. Roughly speaking, an extension of $v$ to $L$ is defined as a valuation $w$ of $L$ such that the restriction of $w$ to $K$ is equivalent to $v$. The study of the set of all such extensions is undertaken in the ramification theory of valuations. This treatment was extended by Manis for commutative rings \cite{Man}. In this section, a parallel treatment is given for superrings. 
\medskip

Let $R$ be a subsuperring of $S$. Let $(v, \mathbb{G})$ be a valuation on $R$ and let $(w, \mathbb{H})$ be a valuation on $S$.

\begin{definition} The valuation 
      $w$ is called an \textit{extension} of $v$ to $S$, if there is an order isomorphism $h:\mathbb{G}\to\mathbb{J}$, where $\mathbb{J}\subseteq\mathbb{H}$ is an ordered subgroup such that for all $x\in R$, we have $w(x)=(h\circ v)(x)$. 
\end{definition}

The following proposition provides equivalent conditions for a valuation $w$ to be an extension of $v$.

\begin{proposition}\label{Proposition:1:3:19}
\it Let $R$ be a subsuperring of $S$. Let $(v, \mathbb{G})$ be a valuation on $R$ and let $(w, \mathbb{H})$ be a valuation on $S$. The following conditions are equivalent:

    \begin{itemize}[itemsep=2pt]
        \item[\rm i)] $w$ is an extension of $v$ to $S$. 
        \item[\rm ii)] $(A_v, \p_v)\preceq(A_w, \p_w)$ and $w|_R$ is a valuation on $R$. 
        \item[\rm iii)] $(A_v, \p_v)\preceq(A_w, \p_w)$ and $\Y(v)\subseteq \Y(w)$.  
    \end{itemize}
\end{proposition}

\begin{proof} Proceed as in \cite[Proposition 3.2]{IB}.
\end{proof}

We will now present a proposition that outlines the conditions under which it is possible to extend a valuation from a superring $R$ to a superring $S$ containing $R$. 

\begin{proposition}[Extension Theorem]\label{Proposition:1:3:20}  \it   Let $R$ be a subsuperring of $S$ and let $(v, \mathbb{G})$ be a valuation on $R$. The valuation $v$ has extensions to $S$ if and only if  $R\cap \langle\Y(v)\rangle_S=\Y(v)$.
\end{proposition}
 
\begin{proof}  Suppose that $R\cap \langle\Y(v)\rangle_S=\Y(v)$ and we will prove that $v$ admits an extension to $S$. For this, we start defining $\q=\p_{v}+ S\cdot\Y(v)$ and $A=A_{v}+S\cdot\Y(v)$. It is easy to check that $A$ is a superring and $\q$ is a superideal of $A$. Moreover, $A_{v}=A\cap R$ and $\p_{v}=\q\cap R=\q\cap A_{v}$. In particular, $\q\cap(A_{v} \setminus \p_{v})=\emptyset$. By \Cref{Super:Krull}, there is some prime ideal $\p$ of $A$, such that $\q\subseteq\p$ and $\p\cap(A_{v} \setminus \p_{v})=\emptyset$; i.e., $\p_{v}=A_{v}\cap \p$ and $A_{v}\subseteq A$. Now, if $(A, \p)$ is a maximal pair, then it is also a valuation pair of the form $(A_w, \p_w)$. On contrary, we can found a maximal pair $(A_w, \p_w)$ dominating it. In any case, we obtain a valuation pair $(A_w, \p_w)$ that dominates $(A_{v}, \p_{v})$. From the definition of $A$, the choice of $A_w$ and $1\in S$, it follows that $\Y(v)\subseteq S\cdot\Y(v)\subseteq A\subseteq A_w$, that is, $S\cdot\Y(v)\subseteq \Y(w)$. This implies $R\cap(S\cdot\Y(v))=\Y(v)\subseteq \Y(w)\cap R$. Thus, we have shown that $\Y(w)\cap R\subseteq \p_w\cap R=\p_{v}$. In conclusion, the elements of $R$ mapped to $ \infty$ by $w$ form a subset of $\p_v$. However, since $w\gg v$, the only possibility is that $\Y(w)\cap R\subseteq \Y(v)$, hence $\Y(w)\cap R=\Y(v)$ and $w$ satisfies the conditions of \Cref{Proposition:1:3:19} iii). Therefore, $w$ extends $v$ to $S$ as claimed.  

Conversely, suppose that $v$ has some extension $w$ to $S$. Thus, by \Cref{Proposition:1:3:19} iii), we see that $\Y(v)\subseteq \Y(w)$, so $R\cap (S\cdot\Y(v))\subseteq R\cap (S\cdot\Y(w))= R\cap \Y(w)=\Y(v)$, hence we find that $R\cap (S\cdot\Y(v))=\Y(v)$. 
\end{proof}

Consider a split superfield $K$ with superreduction $\Bbbk$, so $K=\Bbbk\oplus\J_K$. Let $v$ be a valuation on $\Bbbk$. Can we extend $v$ to $K$? Since $\Y(v)=\{0\}$ and $\Bbbk\cap\langle\Y(v)\rangle_K=\Y(v)$, $v$ can be extended to $K$, by \Cref{Proposition:1:3:20}. A natural extension is defined by $w(x):=v(x+\J_K)$, for any $x\in K$. Are there other ways to extend $v$ to $K$? To address this, we introduce the following definition and proposition.

\begin{definition}
Let $R$ be a subsuperring of $S$.  Let $(w, \mathbb{H})$ be a valuation extending $(v, \mathbb{G})$ from $R$ to $S$.

    \begin{itemize}[itemsep=2pt]
        \item[i)] We define  $(\mathbb{H}/\mathbb{G})_\infty:=(\mathbb{H}/\mathbb{G})\cup\{ \infty\}$.
        
        \item[ii)] We say that $\mathbb{H}/\mathbb{G}$ is \textit{torsion} if for each $\alpha\in\mathbb{H}$ there is some $n\in\Z^+$ with $n\alpha\in\mathbb{G}$.
    \end{itemize}  
\end{definition}

\begin{proposition}\label{Proposition:4.11}
\it  Let $R$ be a subsuperring of $S$. Let $(v, \mathbb{G})$ and $(w, \mathbb{H})$ be valuations on $R$ with $w\gg v$.

    \begin{itemize} 
         \item[\rm i)] The valuation $v$ has an extension to $S$ if and only if $w$ has an extension to $S$.

         \item[\rm ii)] If $(u, \mathbb{G}_u)$  extends  $v$ to $S$, then there is an extension $(t, \mathbb{G}_t)$ of $w$ to $S$ with $t\gg u$.
         
         \item[\rm iii)] If in {\rm ii)} above, the quotient  $\mathbb{G}_u/\mathbb{G}$ is torsion, then $t$ is unique. 
     \end{itemize}  
\end{proposition}

\begin{proof} This is similar to what is found in the purely even case \cite[Proposition 3.5]{IB}.
\end{proof}

\subsection{Integrality and extensions of valuations}

\begin{proposition}\label{Proposition:4.12}
\it    If $(A, \p)$ is a valuation pair of $R$, then $A$ is integrally closed in $R$. 
\end{proposition}

\begin{proof} We only need to show that $\cl_R(A)\subseteq A$, as the inclusion $A\subseteq \cl_R(A)$ generally holds. Furthermore, since $(A,\p)$  is a valuation pair, we have the inclusions $\J_R\subseteq \p\subseteq A$. In particular, $R\od \subseteq A$. Therefore, to show that $A$ is integrally closed in $R$, it suffices to prove that any element $x \in R\ev \setminus R\od^2$ integral over $A$ is in $A$. Indeed, for any $x \in R\ev \setminus R\od^2$, by  \Cref{Proposition:2.12}, we have an equation
\[
x^n + \sum_{0 \leq i \leq n-1} a_i x^i = 0, \quad \text{ where } \quad a_0, a_1, \ldots, a_{n-1} \in A.
\] 

\noindent Let $v$ be a valuation on $R$ such that $(A, \p)=(A_v, \p_v)$. If $v$ is trivial, $A=R$ and we are done. Thus, suppose that $v$ is nontrivial. If $x\in A_v$, we have noting to show. For a contradiction, suppose $x\not\in A_v$. Then $v(x)<0$. Since $v(a_i)\geq0$ and $v(x^i)=iv(x)>nv(x)=v(x^n)$ for all $i=0, 1, \ldots, n-1$, we obtain that $v(a_ix^i)>v(x^n)$, for all $i=0, 1, \ldots, n-1$. On the other hand,  

\begin{align*}
    v(x^n)&=v\left(-\sum_{0\leq i\leq n-1}a_ix^i\right)
    =v\left(\sum_{0\leq i\leq n-1}a_ix^i\right)
    \geq\min_{0\leq i\leq n-1}v(a_ix^i)
    >v(x^n),  
\end{align*}

\noindent which is a contradiction. Thus, $x\in A_v$. For an arbitrary $x\in R$, note that if $x$ is integral over $A$, then $x\od$ and $x\ev$ are so. Thus, $x\ev, x\od\in A$ and since $A$ is a subsuperring of $R$, $x=x\ev+x\od\in A$, which completes the proof of the proposition.
\end{proof}
 
\begin{proposition}\label{prop.3.13.ct}
\it  Let $R$ be a subsuperring of $S$ and $v$ a valuation on $R$.  Suppose that $S$ is integral over $R$. If $w$ is a valuation on $S$ such that the pair $(A_w, \p_w)$ dominates $(A_v, \p_v)$,  then $w$ extends $v$ to $S$. In particular, every valuation $v$ on $R$ has extensions to $S$ if $S$ is integral over $R$, i.e., $\cl_S(R)=R$. 
\end{proposition}

\begin{proof}
Since $(A_v, \p_v)\preceq(A_w, \p_w)$, by \Cref{Proposition:1:3:19} iii), it suffices to prove that   $\Y(v)\subseteq\Y(w)$. Consider $y\in \Y(v)\ev$ and $x\in S\ev$. Since $S$ is integral over $R$, there are $a_i\in R$ with 
    
    \[
    x^n+\sum_{0\leq i\leq n-1}a_ix^i=0,
    \]
    
    \noindent by \Cref{Proposition:2.12}. Then,
    
    \[
    0=\left(x^n+\sum_{0\leq i\leq n-1}a_ix^i\right)y^n=(xy)^n+\sum_{0\leq i\leq n-1}a_i y^{n-i}(x y)^i.
    \]
    
    If $i<n$, we have $v(a_iy^{n-i})=v(a_i)+v(y^{n-i})=v(a_i)+ \infty= \infty$, that is, $b_i:=a_iy^{n-i}$ belongs to $\Y(v)\subseteq A_w$. In other words, 

    \[
    0=(xy)^n+\sum_{0\leq i\leq n-1}b_i(x y)^i,\quad\text{ where }b_0, \ldots, b_{n-1}\in A_w.
    \]
    
    By  \Cref{Proposition:2.12}, $xy$ is integral over $A_w$. Since $A_w$ is integrally closed in $S$, by \Cref{Proposition:4.12}, the element  $yx\in A_w$. Thus, $S\ev(\Y(v))\ev\subseteq A_w$ and $S\ev(\Y(v))\od\subseteq\Y(w)$. Consequently, we see that  $S\ev\,\Y(v)=S\ev(\Y(v))\ev\oplus S\ev(\Y(v))\od\subseteq \Y(w)$.  The inclusions $S\od(\Y(v))\ev\subseteq \Y(w)$ and  $S\od(\Y(v))\od\subseteq \Y(w)$, imply $S\od\,\Y(v)=S\od(\Y(v))\ev\oplus S\od(\Y(v))\od\subseteq \Y(w)$ and thus $S\Y(v)=S\ev\,\Y(v)\oplus S\od\,\Y(v)\subseteq \Y(w)$. Since $\Y(v)\subseteq S\Y(v)$, we have $\Y(v)\subseteq\Y(w)$, meaning $w$ extends $v$ to $S$.
\end{proof} 

The following proposition superizes \cite[Theorem 10.4, p.73]{Mat}.

\begin{proposition}\label{Proposition.5.3.8}
 \it   Let $R\subseteq S\subseteq T$, where $T$ is a superfield and $S=\cl_T(R)$. Then, $S$ is contained in the intersection $W$ of the valuation superrings of $T$ containing $R$. Conversely, if $x\in T\ev \setminus T\od^2$ belongs to any valuation superring of $T$ containing $R$, then $x$ is integral over $R$.
\end{proposition}

\begin{proof} Let $\mathcal{A}$ be the collection of valuation superrings of $T$ containing $R$. If $A_v\in \mathcal{A}$, then $A_v$ is integrally closed by \Cref{Proposition:4.12}. Hence, $S\subseteq A_v$, so $S\subseteq W=\bigcap\limits_{V\in \,\mathcal{A}}V.$ Indeed, $S$ is the smallest integrally closed subsuperring of $T$ containing $R$. To show that $W\subseteq S$, we must see that if $x\in T\ev \setminus T\od^2$ belongs to any valuation superring of $T$ containing $R$, then $x$ is integral over $R$. Equivalently, we must show that if $x\in T$ is not integral over $R$, then there is some valuation superring $A_v$ of $T$ containing $R$ but not $x$. For this, let $x\in T\ev \setminus T\od^2$. Then, $x$ is invertible and $x^{-1}R[x^{-1}]$ is properly contained in $R[x^{-1}]$.
Let $\m$ be a maximal ideal of the superring $R[x^{-1}]$ containing $x^{-1}R[x^{-1}]$. If $(R[x^{-1}], \m)$ is maximal, we are done. Otherwise, let us consider $(A, \p)$ to be a maximal pair containing it. Note that $x^{-1}\in\p$, but $x\not\in A$, as claimed.   
\end{proof}

We say that the superring $R$ is a {\it strong superdomain} if $K(R)\simeq R_{\J_R}$. In this case, the invertible elements in $K(R)$ are precisely the nonzero divisors, $R\setminus \J_R$, which can be identified with $R\ev\setminus R\od^2$.

\begin{corollary}\label{cor:clous}
Let $R$ be a strong superdomain. Then, the  integral closure of $R$ can be decomposed as $\cl(R)=\bigcap\limits_{V\in\, \mathcal{A}}V$, where $\mathcal{A}$ is the collection of valuation superrings of $K(R)$ containing $R$. 
\end{corollary}
  
\begin{proof} This follows immediately from the fact that every element in $K(R)\ev \setminus K(R)\od^2$
is invertible.
\end{proof}


\section{Further properties of valuations}\label{SEC:6}

This section will explore notable properties of valuations on superrings, drawing inspiration from their commutative counterparts. Several of these properties are specifically adapted to extend key aspects of classical valuation theory to the $\Z_2$-graded setting, aiming to generalize concepts established for fields within the broader context of superfields.


\subsection{The inverse property for valuations}   

In this subsection, we examine the inverse property for a set of valuations. This technique was developed by Manis \cite{Man} as a means of addressing the issue of the absence of multiplicative inverses in commutative rings. The proofs of the results presented in this subsection are essentially analogous to those in the commutative case, and no modifications to the arguments are necessary to study its ``super'' version, so such proofs are omitted.   

\begin{definition}
    Let $R$ be a superring and $\L$ a set of valuations on $R$.  We say that $\L$ has the \textit{inverse property} if for every $x\in R$ there is an $x'\in R$ such that $v(xx')=0$ whenever $v\in \L$ and $v(x)\neq \infty$.
\end{definition}

\begin{example}\label{Exam:inv:property}
        Let $K$ be a superfield and $\L$ an arbitrary set of valuations on $K$. Let $x\in K$ and $v\in\L$ such that $v(x)\neq \infty$. Thus, $x\in K \setminus \J_K$ and moreover, $x$ is invertible. It follows that $\L$ has the inverse property. Note that if $\L$ is a singleton and $K$ is not necessarily a superfield, then $\L$ has the  inverse property.  \qed 
\end{example}
 
It is natural to ask ``Under what conditions does a set containing a pair of valuations satisfy the inverse property?'' The following proposition gives necessary and sufficient conditions for a positive answer to this question.

\begin{proposition}\label{P6:6:4} Let $v$ and $w$ be valuations on $R$ such that $\p_v\subseteq\p_w$. Then, $\L=\{v, w\}$ has the inverse property if and only if $A_w\subseteq A_v\cup \Y(w)$.  \qed 
\end{proposition}

The inverse property may be true for sets of the form $\L=\{v, w\}$ that are outside the scope of  \Cref{P6:6:4}, as the following example shows.

\begin{example}\label{other:example} Let $\Z$ be the integers, $p$ and $q$ primes with $p\neq q$. Let $$R=\Z[\,x, x^{-1}\mid \theta_1, \ldots,  \theta_N].$$ Thus, $\overline{R}=\Z[\,x, x^{-1}]$ has some valuation $v_0:\overline{R}\to\mathbb{G}\cup\{ \infty\}$ with valuation pair 

\[
(A, \p)=(\Z[x]+p\overline{R}, xA+p\overline{R}).
\]

\noindent There is a second valuation $w_0:\overline{R}\to\mathbb{H}\cup\{ \infty\}$ with valuation pair 

\[
(B, \q)=(\Z[x^{-1}]+q\overline{R}, x^{-1}B+q\overline{R}).
\]

(See \cite[Example 2 in p.28]{IB}.) Consider the valuations 

\begin{eqnarray*}
    v:R&\mapsto&\mathbb{G}\cup\{ \infty\}\\
    x&\mapsto& v_0(x). 
\end{eqnarray*}

\noindent and

\begin{eqnarray*}
    w:R&\mapsto&\mathbb{H}\cup\{ \infty\}\\
    x&\mapsto&w_0(x).
\end{eqnarray*}

Then, $\Y(v)=p\overline{R}\oplus\J_R$ and $\Y(w)=q\overline{R}\oplus\J_R$. We claim that $\L=\{v, w\}$ has the inverse property. Indeed, if $x\in R$ is such that $v(x)\neq \infty$ and $w(x)\neq \infty$, then $v(x)=v(\overline{x})$ with $\overline{x}\in\overline{R}$. Since $\L=\{v_0, w_0\}$ certainly has the inverse property, there is some $x'\in R$ with $v(xx')=v_0(xx')=0$ and $w(xx)=w_0(xx')=0$, as desired. \qed 
\end{example}

One might expect that if $\L=\{v, w\}$ has the inverse property, then $v$ and $w$ have the same support. However, in \Cref{other:example}, $\Y(v)\neq\Y(w)$. In the following example we see that $\Y(v)=\Y(w)$ does not imply the inverse property for $\L=\{v, w\}$. 

\begin{example}
    Let $R=\Z[\,x, x^{-1}\mid\theta_1, \ldots, \theta_N]$. Consider the valuation $v_0:\overline{R}\to\mathbb{G}\cup\{ \infty\}$ induced by the valuation pair $(A, \p)=(\Z[x], x\Z[x])$ and the valuation  $w_0:\overline{R}\to\mathbb{H}\cup\{ \infty\}$ induced by the valuation pair $(B, \q)=(\Z[x^{-1}], x^{-1}\Z[x^{-1}])$ \cite[Example 3 in p.28]{IB}. If $v$ and $w$ are respectively the  valuations on $R$ corresponding to $v_0$ and $w_0$, then $v$ and $w$ have the same support, namely, $\J_R$. However, $\L_0=\{v_0, w_0\}$ does not satisfies the inverse property, and neither $\L=\{v, w\}$. 
\end{example}

\label{Proposition:5.3.4}\begin{proposition} \it Let $S$ be an extension of $R$ such that $v$ has extensions to $S$. Let $\a\subseteq S$ be a superideal and $\L$ a set of valuations on $S$ extending $v$ from $R$ to $S$. Assume that  

    \[
    \a\subseteq\bigcap_{w\in\L}\Y(w) \quad\text{ and }\quad \a\cap R=\Y(v)\supseteq\J_R.
    \]

    \noindent If $S\ev/\a\ev$ is algebraic over $R/\Y(v)$, then $\L$ has the inverse property. Furthermore, if $$\a=\bigcap\limits_{w\in\L} \Y(w).$$ If $S\ev/\a\ev$ is algebraic over $R/R\cap\a$, then   $\mathbb{H}/\mathbb{G}$ is torsion for all $(w, \mathbb{H})$ in $\L$. \qed 
\end{proposition}

\subsection{Approximation theorems} 


A ring $R$ is said to have a \textit{large Jacobson radical} $\mathfrak{r}_R$ if every prime ideal $\p$ contained in $\mathfrak{r}_R$ is maximal. It is well known that (Manis) valuations behave well in commutative rings with such property \cite[p.1]{KV}. For example, approximation theorems for valuations in such contexts can be proved without additional assumptions \cite{Ara}. It would be interesting to study the behavior of valuations in superrings with a large Jacobson radical. This problem is especially interesting in the case when $R$ is not a superdomain, since if $R$ is a superdomain with a large Jacobson radical, it follows that $R$ is a superfield. This task seems to be tedious, but from what has been studied so far in this paper, we can be sure that it should be possible to carry it out. However, let us think about reducing ourselves a little more, in the interest of making the work simpler and placing our work as close as possible to a geometric context. At first, we are interested in seeing what happens with valuations in superfields, especially when discrete valuations are involved so that we can create an appropriate algebraic environment in which someone can somehow understand the abstract notion of a supercurve. There are three important points about superrings and valuations that we should consider for this task. First of all, we want to deal with split superrings, and particularly, on split superfields. Any set of valuations on a superfield has the inverse property, as we quoted above. On the other hand, if we are dealing with discrete valuations, the quotients of the form $\Z/n\Z$ are torsion. In accordance with this, it will suffice to study the following situation:
\medskip

\textbf{Assumption.} Let $R$ and $S$ be superrings such that $S$ is an extension of $R$. We consider a valuation $(v, \mathbb{G})$ on $R$ with extensions to $S$ and $\L$ is a set of extensions of $v$ to $S$ with the inverse property such that the quotient group $\mathbb{H}/\mathbb{G}$ is torsion for every valuation  $(w, \mathbb{H})$ belonging to $\L$. 

Under the previous conventions, approximation theorems have already been proved by Manis in the commutative context \cite{Man}. Also, J. Gräter \cite{Gra} proved approximation theorems without assuming that the quotient groups $\mathbb{H}/\mathbb{G}$ are torsion.
\medskip

In light of the outcomes presented in the preceding subsection and the premises established in the present subsection, the computational procedures outlined in the works of \cite{IB} and \cite{Man} offer approximation theorems for the context of a superrings, with minimal alterations. The statements of such theorems are provided below, without the inclusion of a proof. 

\begin{theorem} Let $(w_1, \mathbb{H}_1), \ldots, (w_n, \mathbb{H}_n)\in \L$ be independent for all $i\neq j$.

\begin{itemize}
    \item {\rm (Approximation Theorem)}  For all $(\alpha_1, \ldots, \alpha_n)\in\mathbb{H}_1\times\cdots\times\mathbb{H}_n$, there is some $x\in S$ with $w_i(x)=\alpha_i$ for all $i$.
    \item {\rm (Strong Approximation Theorem)}  If for $a_i\in S$ we have $w_i(a_i)\neq \infty$ with $i\geq 1$, then there is some $x\in S$ such that 

\[
w_i(x)=w_i(a_i)<w_i(x-a_i)\quad \text{  for all }i=1, 2, \ldots, n.
\]\qed 
\end{itemize}
\end{theorem}

\subsection{The fundamental inequality}  

The purpose of this subsection is to argue the truth of the subsequent inequality concerning the ramification index $e_v$, the relative degree $f_v$, and the extension degree $n_\L$ in the supersetting:

\begin{equation*}
        \sum_{1\leq i\leq n} e_{v_i}f_{v_i}\leq n_\L.
    \end{equation*}

Similar to the case of approximation theorems, the nature of valuation rings and the prime ideal property of $\Y(v)$ in $R$ imply that the proof of this inequality closely follows its commutative counterpart, requiring only minor modifications. Such a proof is Manis' adaptation (\cite[Proposition 3.23]{Man}) to the case of commutative rings of the comparison theorem proved in \cite[\S11, Theorem 19]{SZ}. The proof will not be included here; instead, the necessary concepts will be defined. 

\begin{definition} 
    Let $A$ be a domain, $\Bbbk$ its field of fractions and let $B$ be an extension of $A$. Note that $\U^{-1}B=B\otimes_A \U^{-1}A=B\otimes_A\K$, with $\U=\Bbbk^*$, is endowed with a natural structure of $\Bbbk$-module. Thus, $\U^{-1}B$ is a $\Bbbk$-vector space. The \textit{rank of $B$ over $A$} is defined as the number 

    \begin{equation}
        \rank(B; A):=\dim_\Bbbk(\U^{-1}B).
    \end{equation}   
\end{definition}

Now we may define the ingredients we require to state the fundamental inequality. 
 \begin{definition}\label{Def:6.8.4}
     Let $S$ be an extension of $R$ and $v$ a valuation on $R$ with extensions to $S$. Consider the set $\L_v=\{w\mid w\text{ extends $v$ to }S\}$ and for any $\L\subseteq\L_v$ define
     
     \[
     \Y(\L):=\bigcap_{w\in\L}\Y(w).
     \]
     
     \noindent Since $\Y(v)$ is a prime ideal of $R$, then $R/\Y(v)$ is an integral domain and the (commutative) ring $S/\Y({\L})$ is an extension of $R/\Y(v)$. Thus, we can consider the number     
    
    \begin{equation}
        n_{\L}:=\rank(S/\Y({\L}); R/\Y(v)).
    \end{equation}

 \noindent For each $w\in\L$, \textit{the relative degree of $w$ with respect to} $v$ is the number 

    \begin{equation}
        f_w:=\rank(A_w/\p_w; A_v/\p_v).
    \end{equation}
   \end{definition}

    \noindent The index of the group $\mathbb{G}$ in $\mathbb{H}$, 

\begin{equation}
    e_w:=(\mathbb{G}:\mathbb{H}),
\end{equation}

\noindent  is called \textit{the reduced ramification index of $w$ with respect to} $v$.
\\

The following proposition can be derived as \cite[Proposition 3.23]{Man}.

\begin{proposition}
\it    Suppose that $S$ is an extension of $R$, $\L\subseteq\L_v$ and $n_{\L}<\infty$ and $\p_v\not\subseteq\p_w$ whenever $v, w$ are independent elements in $\L$. Then, if $v_1, \ldots, v_n\in\L$ are all distinct, one has 

    \begin{equation}
        \sum_{1\leq i\leq n}e_{v_i}f_{v_i}\leq n_\L.
    \end{equation}

    \noindent In particular, $\L$ is finite. \qed 
\end{proposition}

\section{The geometry of valuations}\label{SEC:7} 

This section aims to initiate an investigation into selected geometric constructions using the valuation theory developed herein.

\subsection{Zariski-Riemann superspaces} 

\begin{definition}\label{Def.Zar}
    Let $L$ be a superfield and $K$ be a subsuperring of $L$. We write $\Zar(L, K)$ for the set of valuation superrings $A_v$ on $L$ such that $\mathbb{G}\neq0$ and $K\subseteq A_v$. We will also consider the trivial valuation superring $L$ to be in $\Zar(L, K)$. We call $\Zar(L, K)$ the \textit{Zariski-Riemann superspace} of $L\supseteq K$.
\end{definition}

Observe that if $(A, \p)$ and $(A, \q)$ are valuation pairs, then $\p=\q$. Hence, in \Cref{Def.Zar} it is correct to refer to the valuation superring and not to the valuation pair.

Recall from \Cref{REMAKR:6.1.6} that equivalent valuations define the same valuation superring and \textit{vice versa}. Thus, $\Zar(L, K)$ can also be regarded as the set of equivalence classes of nontrivial valuations on $L$ whose valuation ring contains the superring $K$.

\begin{remark}
   If in \Cref{Def.Zar} we consider $L$ to be a field, then we recover the usual Zariski-Riemann space of $L\mid K$. The Zariski topology on $\Zar(L, K)$ has a basis with basic open sets of the form $U(X_n)$ defined as follows. If $x_1, \ldots, x_n\in L$, we define $\U(X_n)=\Zar(L, K[\, x_1, \ldots, x_n])$. This also defines a topology in the super case. Namely, let $L$ be a superfield and $K$ a subsuperring of $L$. For $x_1, \ldots, x_n\in L\ev$ and $y_1, \ldots, y_m\in L\od$, we define 

\[
\U(X_n\mid  Y_m)=\Zar(L, K[\, x_1, \ldots, x_n\mid  y_1, \ldots,  y_m]),
\]

\noindent that is, $\U(X_n\mid  Y_m)$ is formed by the valuation rings $A_v$ of $L$ containing $K[x_1, \ldots, x_n\mid  y_1, \ldots,  y_m]$ and such that $v$ is nontrivial. Observe that  $\U(X_n\mid  Y_m)=U(X_n\mid Z_r)$, for any $z_1, \ldots, z_r\in L\od.$ Hence, we use $\U(X_n)$ instead of $\U(X_n\mid Y_m)$.  Note that $$\U(X_n)\cap \U(X'_m)=\U(X_n, X'_m), \quad\text{ for any }x_1, \ldots, x_n, x'_1, \ldots, x'_m\in L\ev.$$ Thus, the collection $\mathcal{F}=\left\{\,\U(X_n)\mid n\geq0, x_1, \ldots, x_n\in L\ev\right\}$ gives a basis of open sets for a topology on $\Zar(L, K)$, which we refer as the \textit{Zariski topology}.  
\end{remark}
 
\begin{proposition}\label{prop:7.4}
   $\Zar(L, K)$ is homeomorphic to $\Zar(\overline{L}, \overline{K})$.
\end{proposition}

\begin{proof}
     Consider any ``point'' $A_v\in\Zar(L, K)$, other than $L$. Let $v:L\to\mathbb{G}_\infty$ be the corresponding non-trivial valuation. Since $L$ is a superfield, then $\Y(v)=\J_L$. Thus,  $k(\Y(v))=L/\Y(v)=\overline{L}$. Let 
     $\hat{v}:\overline{L}\to\mathbb{G}_\infty, x+\Y(v)\mapsto v(x)$ be the induced valuation of $v$ on $\overline{L}$. Then $\hat{v}$ is non-trivial and $\overline{L}\subseteq A_{\hat{v}}$. Note that if $A_v=A_w$, then $v$ and $w$ are equivalent and $\hat{v}$ is equivalent to $\hat{w}$, so $A_{\hat{v}}=A_{\hat{w}}$. Thus,  
     
     \[
     \psi:\Zar(L, K)\to\Zar(\overline{L}, \overline{K}), A_v\mapsto A_{\hat{v}}
     \]
     is well defined. Now assume that $A_v, A_w\in\Zar(L, K)$ are such that $A_{\hat{v}}=A_{\hat{w}}$. Then, $\hat{v}=\hat{w}\circ\phi$, where $\phi:\mathbb{H}\to\mathbb{G}$ is an order isomorphism. By the assumptions, $v$ extends $\hat{v}$ to $L$ and $w\circ \phi$ extends $\hat{w}\circ \phi$ to $L$. Observe that $v=w\circ\phi$, so $A_v=A_w$ and the correspondence is one-to-one. Now, if $v_0:\overline{L}\to\mathbb{G}_\infty$ is such that $A_{v_0}\in\Zar(\overline{L}, \overline{K})$, then we may extend $v_0$ to $L$ putting $v(x)=v_0(\overline{x})$ for all $x\in L$. It follows that $K\subseteq A_v$ and $A_v\in\Zar(L, K)$ is mapped to $A_{v_0}=A_{\hat{v}}$. Hence, $\psi$ is bijective. Finally, $\psi(U(X_n))=U(X'_n)$, with $x'_i=x_i+\J_L$. Therefore $\psi$ is an homeomorphism, because it is a continuous open bijection.
\end{proof}

By \cite[\S4, Corollary 1 to Theorem 4 and Corollary 1 to Theorem 5]{SZ}, the only case in which the set $\Zar(\overline{L}, \overline{K})$ is $\{\overline{L}\}$ is when $\overline{K}$ is a field and $\overline{L}$ is an algebraic extension of the field $\overline{K}$. Thus, $\Zar(L, K)$ is $\{L\}$ only in that case, which we will exclude in our analysis.
\\

We now define a sheaf $\O_X$ on $X=\Zar(L, K)$. For any nonempty open set $U\subseteq X$, we define the superring 

\begin{equation}\label{eq:def:OX}
    \O_X(U)=\bigcap_{A_v\in U}A_v\quad \text{and}\quad\O_X(\emptyset)=0.
\end{equation}
The superring $\O_X(U)$ is called the \textit{coordinate superring} of $U$. Since  $U\subseteq V$ implies $\O_X(V)\subseteq\O_X(U)$,  the restriction mappings are given by the inclusions. Observe that $\O_{X, A_v}=A_v$, so $(X, \O_X)$ is a locally ringed superspace, by \Cref{equiv:v:local}.  

\begin{problem}
     Develop a theory of Zariski-Riemann superspaces and compare this theory with its purely even counterpart. 
\end{problem}

\begin{problem}
    Develop a theory of birrational supergeometry and study if it may be modeled via Zariski-Riemann superspaces (cf. \cite{PM}).
\end{problem}

\subsection{Abstract arithmetically non-singular supercurves}  

Any superfield considered in this subsection is supposed to be split (\Cref{def.sev.cond} v)). 

Consider $L\mid K$ a superfields extension (i.e., $L\supseteq K$). Discrete valuations on $L$ that are trivial on $K$ are called {\it valuations of} $L\mid K$.  

Let $K$ be a superfield and $v:K\to\Z_\infty$ a discrete valuation. Let $\Bbbk$ be the superreduction of $K$. By \Cref{vhat:vline:are:valuations}, we have a valuation $\hat{v}:\Bbbk=K/\Y(v)\to\Z_\infty$ induced by $v$. In a similar way, if we have a valuation $v_0:\Bbbk\to\Z_\infty$, then it induces a valuation $v$ on $K$ such that $\hat{v}=v_0$. Thus, we may identify $v$ and $\hat{v}$, as well as the pairs $(A_v, \p_v)$ and $(\mathfrak{o}_v, \m_v):=(A_{\hat{v}}, \p_{\hat{v}})$ may be identified. 

Now assume that $\overline{L}$ is a function field of dimension one over $\Bbbk:=\overline{K}$, i.e., it is a finitely generated extension field of transcendence degree one. In particular, we are assuming on $L$ and on $K$ the $\Bbbk$-superalgebra structure induced by $\Bbbk$, which is assumed to be algebraically closed. 

Let $C_L$ be the set of valuation superrings of discrete valuations of $L\mid K$ and define $C_{\overline{L}}$ in a similar way. Note that if the underlying superfields have nontrivial odd part, then $C_{\overline{L}}\neq C_L$, but we have a one to one correspondence between these sets. We introduce the cofinite topology on $C_{\overline{L}}$. Topologize $C_L$ through  the bijection $\psi:C_L\to C_{\overline{L}}, A_v\to\mathfrak{o}_v$. Thus, we have a natural homeomorphism  $C_{\overline{L}}\simeq C_L$. It is a folklore that $X_{\mathrm{ev}}:=C_{\overline{L}}$ is infinite, so $X:=C_L$ is so.

Consider the structure sheaf $\O_X$ on $X$ given by \eqref{eq:def:OX}. This is a sheaf of $\Bbbk$-superalgebras. Also, note that if  $\O_{X_{\mathrm{ev}}}$ is defined in the obvious way, then for all open $U$, we have $
\overline{\O_{X}(U)}=\O_{X_{\mathrm{ev}}}(U)$, where $U$ is identified with its image in $X_{\mathrm{ev}}$. Furthermore, $\O_X(U)=\O_{X_{\mathrm{ev}}}(U)\oplus\mathcal{J}_{\O_{X}(U)}$. 

Now consider $s\in\O_{X_{\mathrm{ev}}}(U)$. By classical theory, $s$ can be viewed as a function $s:U\to\Bbbk$. Note that $t\in\J_{\O_{X}(U)}$ is a section that is non-zero, but evaluates to zero at every point. The section $s+t:U\to\Bbbk$ is defined in the obvious way. Furthermore, if $f, g:U\to\Bbbk$ define the same function, then $f-g\in\J_{\O_X(U)}$. Hence, every element in $\O_{X_{\mathrm{ev}}}(U)$  can be identified with a function $U\to\Bbbk$. Now, let $a\in\overline{L}$. Then, there exists some $p\in X_{\mathrm{ev}}$ such that $a\in p$. Therefore, every element of $L$ is some point of $X$. Thus, every element of $L$ is in some $\O_X(U)$. 

Recall that the space $X$ is irreducible, Noetherian, and have closed points. We define $K(C_L)$ to be the set of pairs of the form $(U, f)$, where $U$ is a nonempty subset of $X$ and $f\in\O_X(U)$. Consider the equivalence relation $(U, f)\sim (V, g)$ if there exists $W\subseteq U\cap V$ such that $\overline{f}|_{W}=\overline{g}|_{W}$. Since every $\O_X(U)$ is contained in $L$, $K(C_L)$ has a natural structure of superring induced by $L$. The superfield of functions of $C_L$ is $L$, because every element of $L$ is in some $\O_X(U)$. 

\begin{remark}
It seems reasonable to expect that realizable discrete valuation superrings could be non-Noetherian and that $(X, \mathcal{O}_X)$ constructed above could have superdimension $1|\infty$; that is, for an affine $U$, it might occur that $\mathcal{O}_X(U)_{\bar{1}}$ is not finitely generated as an $\mathcal{O}_X(U)_{\bar{0}}$-module, while the superreduced of $(X, \mathcal{O}_X)$ has dimension one. This implies that regularity, and consequently non-singularity (in the sense of \cite{MZ}), might be unattainable in this context, except in the trivial  case. To address this issue, we will henceforth restrict our consideration to cases where Noetherianity holds. Therefore, our superrings will be Noetherian, integrally closed, and of even Krull superdimension one. However, this is insufficient to create a theory completely analogous to the classical one. Indeed, regularity (and hence non-singularity) does not hold unless the superrings have a trivial odd part \cite[Remark 5.12]{RTT}. We thus use the term \textit{arithmetically regular} for a superring that is integrally closed, Noetherian, and of even Krull superdimension one.
\end{remark}

\begin{definition}
    An \textit{abstract arithmetically non-singular supercurve} is any open subset of $X$ with the restricted sheaf of $\O_X$.  An \textit{abstract non-singular curve} is any open subset of $X_{\mathrm{ev}}$ with the obvious sheaf.  
\end{definition}

The selection of the surname ``arithmetically'' is based on the fact that the rings in question display an excellent arithmetical tendency with regard to their superideals. For example, the prime factorization of superideals is a consequence of this property, among other results. It should be noted that our construction is not merely non-singular in the sense of \cite{MZ}, unless we are working in the purely even case. The model we have constructed is relatively large, in the sense that the rings have sufficient coefficients to be normal, a property that regular superrings lack.

\begin{problem}
Develop a theory of arithmetically non-singular supercurves.
\end{problem}

A natural observation from comparing non-singular and arithmetically non-singular supercurves is as follows. By \cite[Proposition 6.9]{RTT}, the stalks $\mathcal{O}_{X,x}$ of a (non purely even) non-singular supercurve $(X, \mathcal{O}_X)$ are superdomains for all $x\in X$. By \Cref{Proposition.5.3.8}, $\mathcal{O}_{X, x}\subsetneq\mathrm{cl}(\mathcal{O}_{X, x})\subseteq\bigcap_{R\in \mathcal{S}}R$, where $\mathcal{S}$ is the collection of valuation superrings of $\mathcal{O}_{X,\eta}$ containing $\mathcal{O}_{X,x}$ (with $\eta$ the generic point). However, for an arithmetically non singular supercurve $(X, \mathcal{O}_X)$, we find that $\mathcal{O}_{X, x}=\mathrm{cl}(\mathcal{O}_{X,x}) = \bigcap_{R \in \mathcal{S}} R$. As a result, a substantial difference is anticipated between the theory developed for arithmetically non-singular supercurves and that of \cite{RTT}, specifically when the odd part is non-trivial.

\section*{Acknowledgements}
P. Rizzo and J. Torres del Valle gratefully acknowledge partial support from CODI (Universidad de Antioquia, UdeA) through project numbers 2020-33713, 2022-52654 and 2023-62291.

A. Torres-Gomez is grateful to the Shanghai Institute for Mathematics and Interdisciplinary Sciences (SIMIS, China) for their hospitality and partial financial support during his visit in December 2024 and January 2025, where a part of this project was completed.

\end{document}